\documentclass[10pt]{amsart}
\usepackage{amsthm}
\usepackage{graphicx}
\usepackage{color}

\newtheorem{theorem}{Theorem}[section]
\newtheorem{lemma}[theorem]{Lemma}

\newtheorem{proposition}{Proposition}[section]

\newtheorem{example}{Example}[section]

\usepackage[colorlinks,citecolor=blue,urlcolor=blue,bookmarks=false,hypertexnames=true]{hyperref}

\usepackage{enumerate,multirow,cite}

\begin{document}

\title[Local fractional metric dimension of symmetric planar graphs]{Local fractional metric dimension of rotationally symmetric planar graphs arisen from planar chorded cycles}

\author{Shahbaz Ali}
\address{Department of Mathematics, University of the Punjab, Lahore, Pakistan}
\email{shahbaz.math@gmail.com}

\author{Ra\'ul M. Falc\'on}
\address{Department of Applied Mathematics I, Universidad de Sevilla, Spain}
\email{rafalgan@us.es}

\author{Muhammad Khalid Mahmood}
\address{Department of Mathematics, Khwaja Fareed University of Engineering \& Information Technology,\\ Rahim Yar Khan, Pakistan.}
\email{khalid.math@pu.edu.pk}



\begin{abstract}
In this paper, a new family of rotationally symmetric planar graphs is described based on an edge coalescence of planar chorded cycles. Their local fractional metric dimension is established for those ones arisen from chorded cycles of order up to six. Their asymptotic behaviour enables us to ensure the existence of new families of rotationally symmetric planar graphs with either constant or bounded local fractional dimension.
\end{abstract}

\maketitle

\section{Introduction}

In the 1970's,  Slater \cite{Slater1975} and Harary and Melter \cite{Harary1976} introduced independently the {\em metric dimension problem}, which consists of determining the minimum number of vertices within a graph that may uniquely be represented by their respective vector of distances. Being NP-hard \cite{Garey1979}, this problem has explicitly been solved for different types of graphs \cite{Caceres2005,Chartrand2000,Imran2013}. General and specific bounds are known depending on the order, maximum degree or diameter of the graph under consideration \cite{Chartrand2000,Feng2012,Slater1988}. In this regard, Imran et al. \cite{Imran2012,Imran2016} asked for characterizing families of (rotationally symmetric) planar graphs with constant metric dimension (see also \cite{Chartrand2000,Javaid2008,Khuller1996}). The metric dimension problem plays a relevant role not only in the study of structural properties of graphs, but also in solving real life problems such as robot navigation \cite{Khuller1996}, pattern recognition and image processing \cite{Melter1984}, representation of chemical compounds \cite{Chartrand2000}, combinatorial optimization \cite{Sebo2004} or networking \cite{Beerliova2005}, amongst others. Particularly, the metric dimension problem concerning hexagonal graphs \cite{Manuel2008,Shreedhar2010,Xu2014} has acquired special relevance because of their implementation in computer graphics \cite{Lester1984}, multiprocessor networks \cite{Chen1990} and cellular networks \cite{Nocetti2002}.

\vspace{0.2cm}

In 2000, Chartrand et al. \cite{Chartrand2000} formulated the metric dimension problem as an integer programming problem. Shortly after, Currie and Oellermann \cite{Currie2001} formulated a linear programming relaxation whose optimal solution they termed {\em fractional metric dimension} of the graph (see also \cite{Arumugan2012,Fehr2006}). This new problem has explicitly been solved for different types of graphs \cite{Arumugan2013,Feng2018,Feng2014,Feng2013,Saputro2018}.

\vspace{0.2cm}

In 2018, a local version of the fractional metric dimension concerning only adjacent vertices was introduced by Benish et al. \cite{Benish2018} (see also \cite{Benish2019}). Even if its study is still in a very initial stage, the local fractional metric dimension has explicitly been determined for several types of graphs \cite{Aisyah2019,Javaid2020}. More recently, Liu et al. \cite{Liu2020} have computed the local fractional metric dimension of a family of rotationally symmetric planar graphs derived from an edge coalescence of a cycle of order $m$ with $m$ distinct chorded cycles of a same order $n\in\{3,4,5\}$. This paper delves into this last topic by describing a new family of rotationally symmetric planar graphs arisen from an edge coalescence of $m$ planar chorded cycles of order $n$.

\vspace{0.2cm}

The paper is organized as follows. In Section~\ref{sec:preliminaries}, we remind some preliminary concepts and results on graph theory that are used throughout the paper. In Section~\ref{sec:family}, we describe the mentioned  family of rotationally symmetric planar graphs. We focus in particular on those ones derived from planar chorded cycles of order $n\in\{4,5,6\}$, for which we study their local fractional metric dimension. The case $n=4$ is dealt with at the end of Section~\ref{sec:family}, whereas exhaustive analyses of the cases $n=5$ and $n=6$ are done in Section \ref{sec:F1}. The obtained results are summarized in the conclusion section, where we make use of the asymptotic behaviour in order to ensure which ones of these rotationally symmetric planar graphs have either a constant or a bounded local fractional metric dimension. Finally, all the tables described in the body of the manuscript are enumerated in Appendix \ref{app}.

\section{Preliminaries}\label{sec:preliminaries}

Let us review some basic concepts and results on graph theory that are used throughout the paper. See \cite{West1996} for more details about this topic.

\vspace{0.1cm}

Any {\em graph} $G$ is formed by a set $V(G)$ of {\em vertices} and a set $E(G)$ of {\em edges} so that each edge contains two vertices, which are said to be {\em adjacent}. The subset of vertices that are adjacent to a given vertex $v\in V(G)$ constitutes its {\em neighborhood} $N(v)$. Two adjacent vertices having the same neighborhood are called {\em true twin vertices}. From here on, we denote $uv$ the edge formed by two vertices $u,v\in V(G)$. The number of vertices and the number of edges of the graph $G$ constitute, respectively, its {\em order} and {\em size}. If both of them are finite, then the graph is said to be {\em finite}. The graph $G$ is called {\em bipartite} if the set $V(G)$ may be partitioned into two subsets so that every edge contains exactly one vertex of each subset. Further, two graphs are {\em isomorphic} if there exists a bijection between their sets of vertices preserving their adjacency.

A {\em path} between two distinct vertices $v,w\in V(G)$ is any ordered sequence of adjacent and pairwise distinct vertices $\langle\,v_0=v,v_1,\ldots,v_{n-2},v_{n-1}=w\,\rangle$ in $V(G)$, with $n>2$. A graph is {\em connected} if there always exists a path between any pair of vertices. If the initial and final vertices of a path coincide, then it is called a {\em cycle}. If all the vertices of a cycle are joined to a new vertex, then the resulting graph is called a {\em wheel}. The new vertex is called the {\em center} of the wheel. Further, a {\em chord} of an existing cycle in $G$ is any edge of the graph containing two non-adjacent vertices of the cycle. A {\em chorded cycle} is any cycle containing at least one chord.

A {\em planar graph} is any graph that can be embedded in the plane. That is, it may be drawn on the plane without crossing edges. Figure \ref{Figure_PCC} illustrates the set of non-isomorphic planar chorded cycles of order $n\leq 6$.

\begin{figure}[ht]
\centering
\includegraphics[scale=0.073]{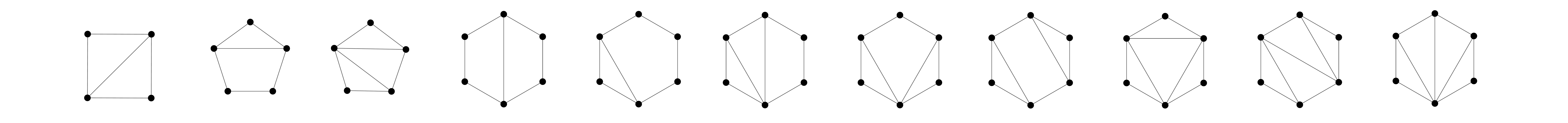}
\caption{Non-isomorphic planar chorded cycles of order $n\leq 6$.}
\label{Figure_PCC}
\end{figure}

The union of two graphs $G_1$ and $G_2$ is the graph $G_1\cup G_2$ such that $V(G_1\cup G_2)=V(G_1)\cup V(G_2)$ and $E(G_1\cup G_2)=E(G_1)\cup E(G_2)$. If $E(G_1)\neq\emptyset\neq E(G_2)$, then the {\em edge coalescence} \cite{Globler2013} of $G_1$ and $G_2$ via $u_1v_1\in E(G_1)$ and $u_2v_2\in E(G_2)$ is the graph $G_1\cdot G_2(u_1v_1,u_2v_2:u_3v_3)$ resulting after identifying in $G_1\cup G_2$ both edges $u_1v_1$ and $u_2v_2$, which merge into a single new edge $u_3v_3$. Figure \ref{Figure_Coalescence} illustrates this concept with three different examples of edge coalescence between the same pair of planar chorded cycles of order six.

\begin{figure}[ht]
\begin{center}
\includegraphics[scale=0.7]{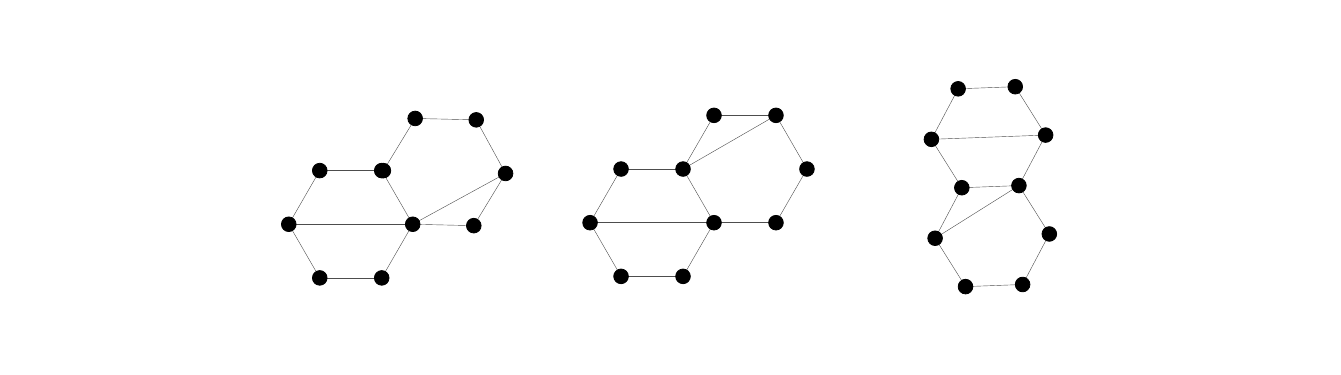}
\caption{Edge coalescences between the same pair of planar chorded cycles of order six.}
\label{Figure_Coalescence}
\end{center}
\end{figure}

Let $G=(V(G),E(G))$ be a finite connected graph. The minimum number of edges within a path between two vertices  $v,w\in V(G)$ constitutes the {\em distance} $d(u,v)$. The {\em diameter} $\mathrm{diam}(G)$ of the graph $G$ is the maximum distance between any two vertices in $V(G)$. Further, the {\em representation} of a vertex $v\in V(G)$ with respect to an ordered subset $S=\{v_1, \ldots, v_k\}\subseteq V(G)$  is the ordered $k$-tuple $r(v|S):=(d(v, v_1), \ldots, d(v, v_k))$. The subset $S$ is a {\em resolving set} for the graph $G$ if $r(v|S)\neq r(w|S)$, for every pair of distinct vertices $v,w\in V(G)$. If such condition holds for every edge $vw\in E(G)$, then the subset $S$ is a {\em local resolving set}. The {\em (local) metric dimension} of the graph $G$ is then defined as the minimum number of vertices contained in any of its (local) resolving sets. It is denoted $\dim(G)$ ($\mathrm{ldim}(G)$ in its local version).

The {\em resolving neighbourhood} of a pair of vertices $v,w\in V(G)$ is the set
\[\mathcal{R}\{v,w\}:=\{u \in V(G) | ~d(v, u) \neq d(w, u)\}.\] From here on, in order to simplify the notation of this manuscript, we also denote  $\overline{\mathcal{R}}\{v,w\}=V(G)\setminus\mathcal{R}\{v,w\}$. A {\em resolving function} \cite{Arumugan2012} of the graph $G$ is any map $\vartheta: V(G) \rightarrow [0, 1]$ such that

\begin{equation}\label{eq_RF} \sum_{u \in \mathcal{R}\{v, w\} } \vartheta(u)\geq 1,
\end{equation}
for every pair of distinct vertices $v,w\in V(G)$. The {\em fractional metric dimension} of the graph $G$ is
\[\dim_{\mathrm{f}}(G):= \min\left\{ \sum_{v \in V(G)} \vartheta(v) \colon\, \vartheta \textrm{ is a resolving function of } G \right\}.\]
The concepts of {\em local resolving neighbourhood} and {\em local resolving function} arise similarly in case of dealing only with pairs of adjacent vertices. Then, the {\em local fractional metric dimension} of the graph $G$ is defined as
\[\mathrm{ldim}_{\mathrm{f}}(G):= \min\left\{ \sum_{v \in V(G)} \vartheta(v) \colon\, \vartheta \textrm{ is a local resolving function of } G \right\}.\]
Further, we denote from now on $\ell(G):=\min\{\left|\mathcal{R}\left\{v,w\right\}\right|\colon\, vw\in E(G)\}$. In particular, since $v,w\in\mathcal{R}\{v,w\}$, for all $v,w\in V(G)$, it is $\ell(G)\geq 2$. The next result follows from all the previous definitions.

\begin{lemma}[\cite{Okamoto2010,Benish2018,Benish2019}]\label{LemmaBenish} Let $G$ be a finite connected graph of order $n\geq 2$. Then,
\begin{equation}\label{eq_Benish1}
\mathrm{ldim}_{\mathrm{f}}(G)\leq\dim_{\mathrm{f}}(G),
\end{equation}
\begin{equation}\label{eq_Benish2}
\frac n{n-\mathrm{ldim}(G)+1}\leq \mathrm{ldim}_{\mathrm{f}}(G)\leq\frac n {\ell(G)}\leq \frac n2
\end{equation}
\begin{equation}\label{eq_Benish3}
1\leq \mathrm{ldim}_{\mathrm{f}}(G)\leq\mathrm{ldim}(G)\leq n-\mathrm{diam}(G).
\end{equation}
In addition, the following assertions are satisfied.
\begin{enumerate}
    \item $\mathrm{ldim}_{\mathrm{f}}(G)=1$ if and only if the graph $G$ is bipartite.
    \vspace{0.1cm}
    \item $\mathrm{ldim}_{\mathrm{f}}(G)=\frac n2$ if and only if each vertex in $V(G)$ has a true twin vertex.
\end{enumerate}
\end{lemma}

\vspace{0.1cm}

\begin{example}\label{Example0}
Let $G$ be the planar graph described in Figure \ref{Figure_Example0}. In order to determine an upper bound of $\mathrm{ldim}_{\mathrm{f}}(G)$, we determine the cardinality of its local resolving neighbourhoods. To this end, the symmetry of the graph $G$ enables us to focus on the following four resolving neighbourhoods.
\[\overline{\mathcal{R}}\{v_1,v_7\}=\{v_{10}, v_{14}, v_{15}, v_{16}, v_{20}, v_{21}\}.\]
\[\overline{\mathcal{R}}\{v_7,v_{13}\}=\{v_{2}, v_{5}, v_{8}, v_{12}\}.\]
\[\overline{\mathcal{R}}\{v_{13},v_{18}\}=\{v_{3}, v_{6}, v_{21}, v_{24}\}.\]
\[\overline{\mathcal{R}}\{v_{13},v_{24}\}=\{v_{4}, v_{5}, v_{10}, v_{11},v_{12}, v_{16}, v_{17}, v_{18}, v_{22}, v_{23}\}.\]
Thus, $\ell(G)=14$ and hence, Condition (\ref{eq_Benish2}) implies that  $\mathrm{ldim}_{\mathrm{f}}(G)\leq \frac{24}{14}=\frac{12}{7}$. In fact, a simple computation establishes that $\mathrm{ldim}_{\mathrm{f}}(G)=\frac 32$.
\end{example}

\begin{figure}[ht]
\begin{center}
{\includegraphics[scale=0.14]{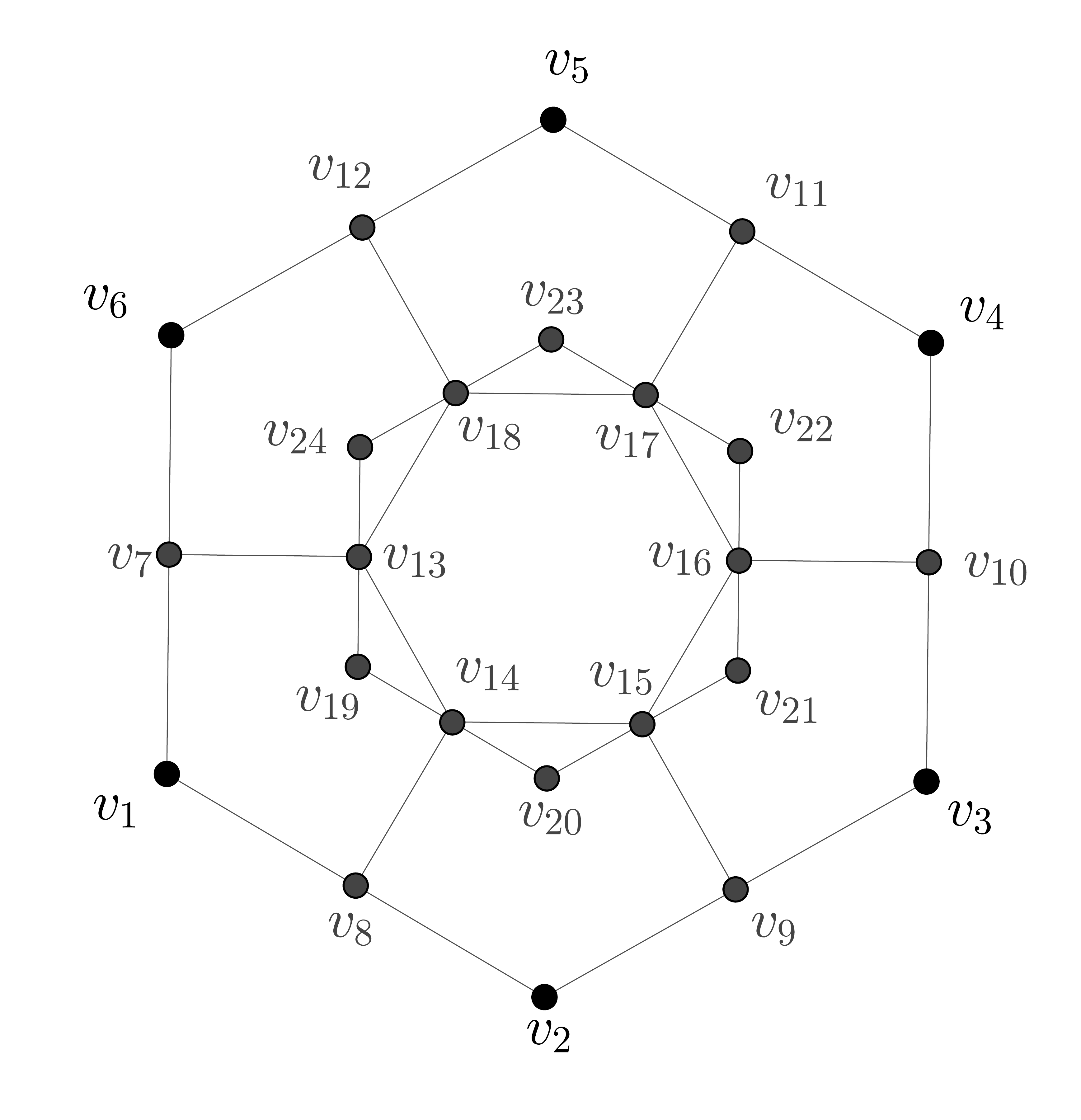}}\caption{Planar graph $G$ such that $\mathrm{ldim}_{\mathrm{f}}(G)=\frac 32\leq \frac {12}7$.}\label{Figure_Example0}
\end{center}
\end{figure}

\section{A family of rotationally symmetric planar graphs arisen from planar chorded cycles}\label{sec:family}

This paper focuses on the asymptotic behaviour of the local fractional metric dimension of a particular family of rotationally symmetric planar graphs. They arise from a sequential edge coalescence among a series of disjoint copies of a given planar chorded cycle. In this section, we detail their construction and establish some basic results concerning their local fractional metric dimension.

\vspace{0.1cm}

Let $G_1,\ldots, G_m$ be $m$ disjoint copies, with $m\geq 2$, of a planar chorded cycle $G$ of order $n$, whose set of vertices is $V(G)=\{v_1,\ldots,v_n\}$. Let $v_i^k\in V(G_k)$ denote the corresponding copy of each vertex $v_i\in V(G)$. Without loss of generality, we assume that the vertices are naturally labeled counterclockwise. Then, we are interested in the planar graph $\mathcal{G}^m(G)$ that is sequentially defined as follows.
\begin{itemize}
    \item Let $\mathcal{G}_1(G):= G_1\cdot G_2(v_2^1v_3^1,v_{n-1}^2v_n^2:v_{n-1}^2v_n^2)$. (Notice that we label the merged new edge in the same way that the edge in the second graph under consideration. The same is done in the subsequent steps.)

    \vspace{0.1cm}

    \item Let  $\mathcal{G}_k(G):= \mathcal{G}_{k-1}(G)\cdot G_{k+1}(v_2^kv_3^k,v_{n-1}^{k+1}v_n^{k+1}:v_{n-1}^{k+1}v_n^{k+1})$, for each positive integer $k\in\{2,\ldots,m-1\}$.

    \vspace{0.1cm}

    \item Finally, let $\mathcal{G}^m(G):=\mathcal{G}_{m-1}(G)\cdot \mathcal{G}_{m-1}(G)(v_2^mv_3^m,v_{n-1}^1v_n^1:v_{n-1}^1v_n^1)$.
\end{itemize}
The resulting graph $\mathcal{G}^m(G)$ is a rotationally symmetric planar graph of order $m\cdot (n-2)$.

\vspace{0.1cm}

Based on the planar chorded cycles described in Figure \ref{Figure_PCC}, we enumerate in Figure \ref{Figure_PG} all the vertex-labeled planar chorded cycles of order $n\leq 6$ on which the just defined constructive procedure may be implemented in order to get non-isomorphic rotationally symmetric planar graphs. From here on, we refer them as {\em quadrilateral} ($Q_1$ and $Q_2$, for $n=4$), {\em pentagonal} ($P_1$ to $P_6$, for $n=5$) or {\em hexagonal} ($H_1$ to $H_{17}$, for $n=6$) {\em chorded cycles}.

\begin{figure}[ht]
\begin{center}
\begin{tabular}{c}
{\includegraphics[scale=0.4]{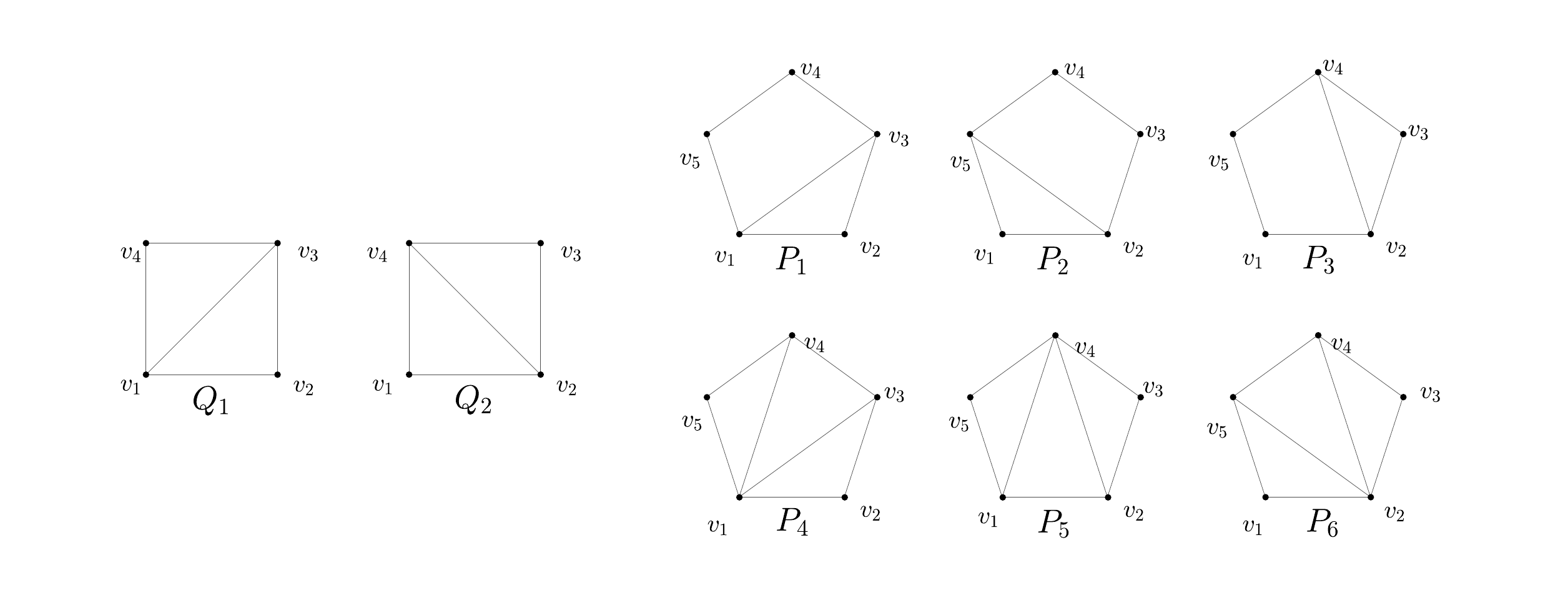}}\\
{\includegraphics[scale=0.52]{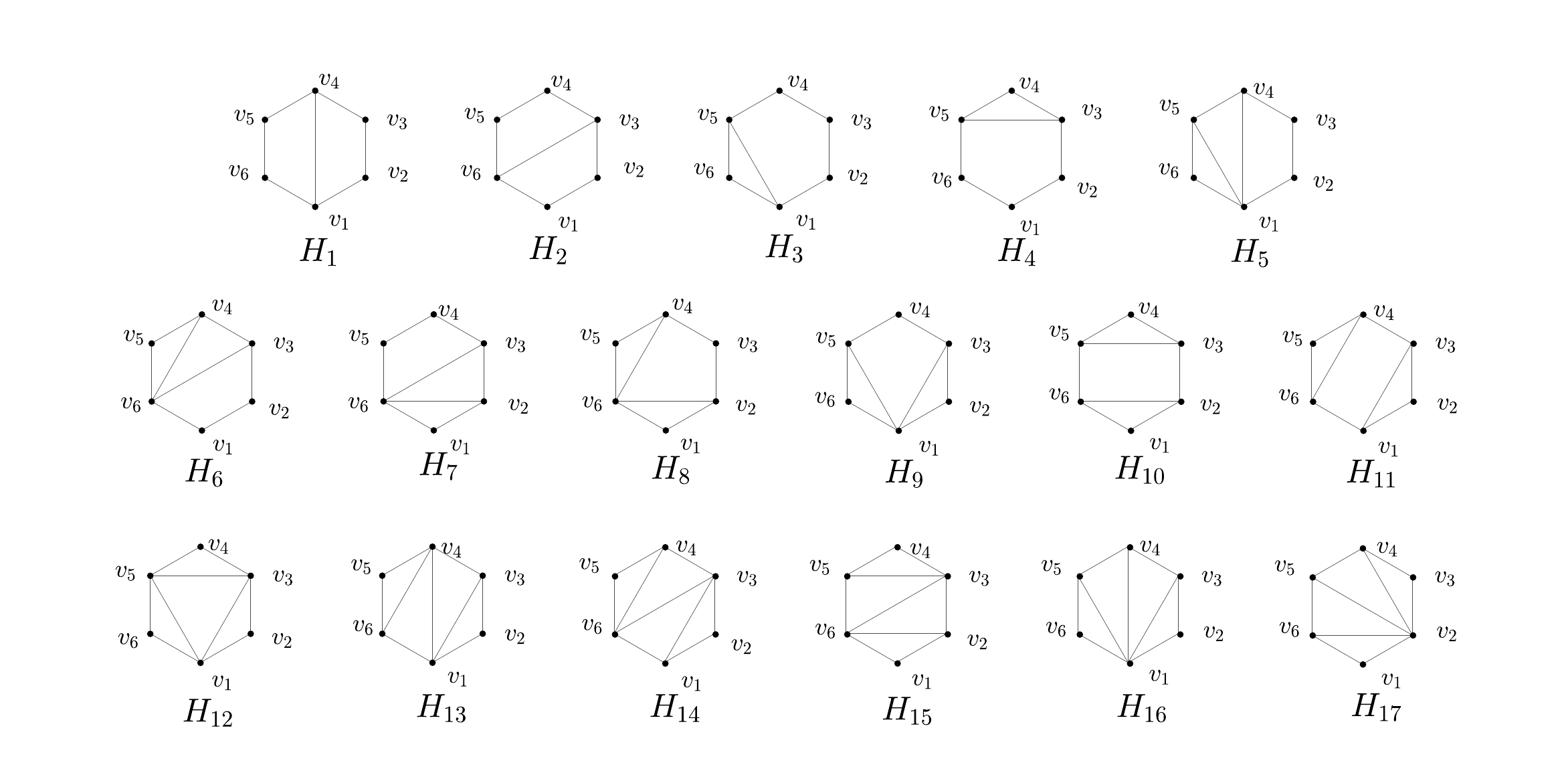}}
\end{tabular}
\caption{Quadrilateral, pentagonal and hexagonal  chorded cycles.}\label{Figure_PG}
\end{center}
\end{figure}

Figure \ref{Figure_RSPG4} illustrates a representation of each one of the two non-isomorphic rotationally symmetric planar graphs arisen from the quadrilateral chorded cycles $Q_1$ and $Q_2$. In addition, Figure \ref{Figure_HnH6} illustrates the representations of a pair of rotationally symmetric planar graphs arisen from the pentagonal chorded cycle $P_4$ and the hexagonal chorded cycle $H_5$. Notice that the first and last vertical edges in each one of these two representations refer to the same edge. In a similar way, Figures \ref{Figure_PG5} and \ref{Figure_RSPG} outline all the rotationally symmetric planar graphs arisen from pentagonal and hexagonal chorded cycles. Notice that, even if the vertex-labeling has been omitted for making clearer the illustration, it coincides in each case with that one described in Figure \ref{Figure_HnH6}.

\begin{figure}[ht]
\begin{center}
{\includegraphics[scale=0.2]{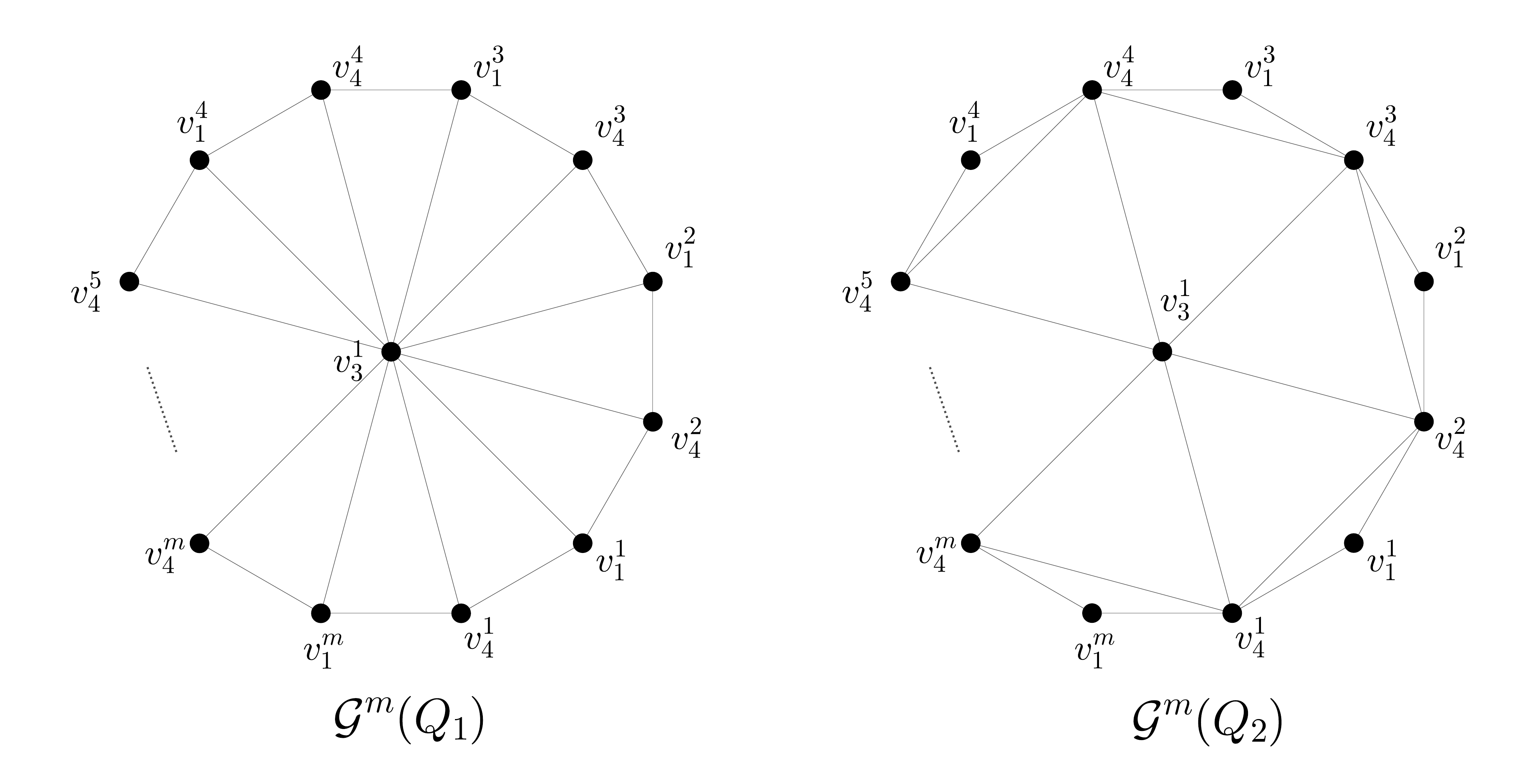}}\caption{The graphs $\mathcal{G}^m(Q_1)$ and $\mathcal{G}^m(Q_2)$.}\label{Figure_RSPG4}
\end{center}
\end{figure}

\begin{figure}[ht]
\begin{center}
{\includegraphics[scale=0.063]{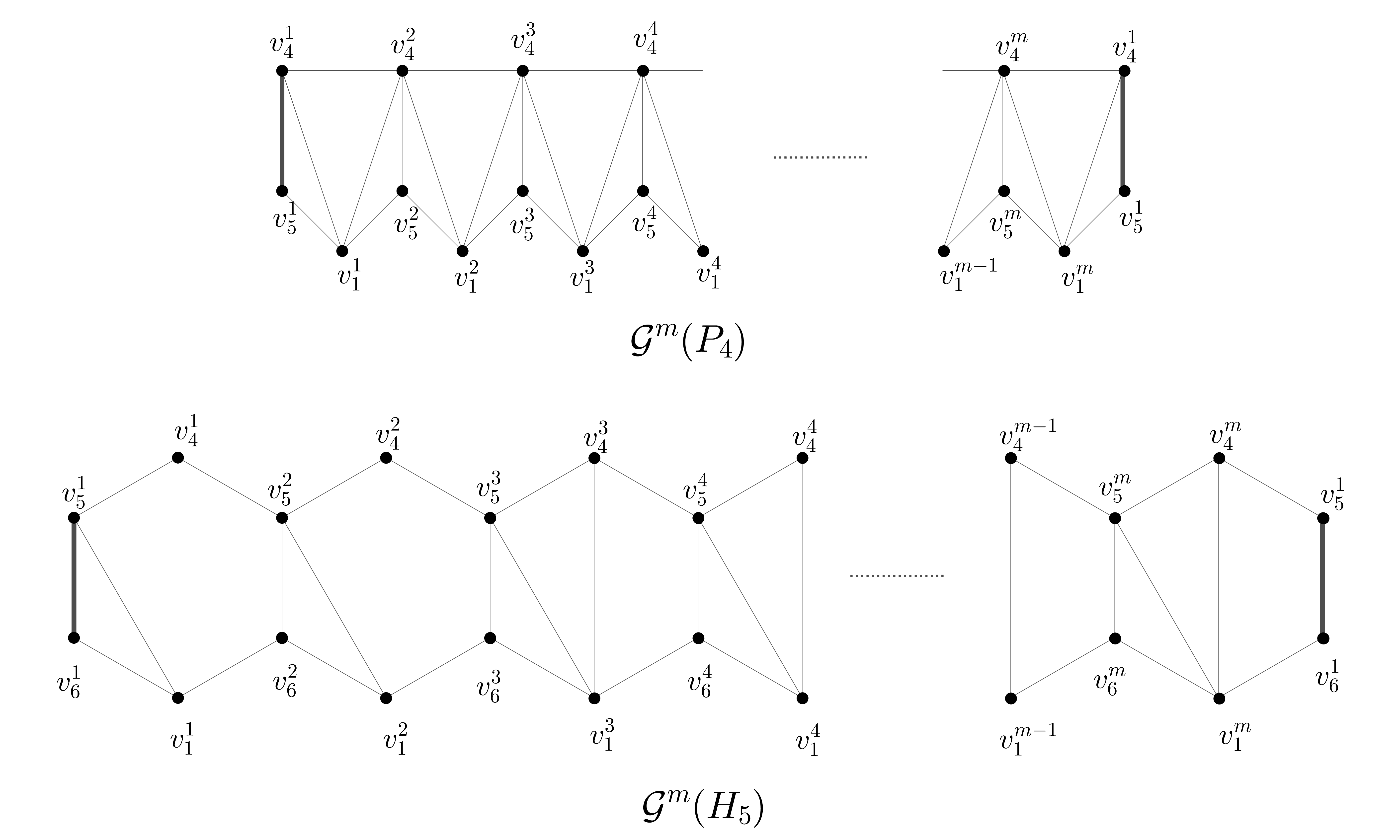}}\caption{The graphs $\mathcal{G}^m(P_4)$ and $\mathcal{G}^m(H_5)$}\label{Figure_HnH6}
\end{center}
\end{figure}

\begin{figure}[ht]
\begin{center}
{\includegraphics[scale=0.061]{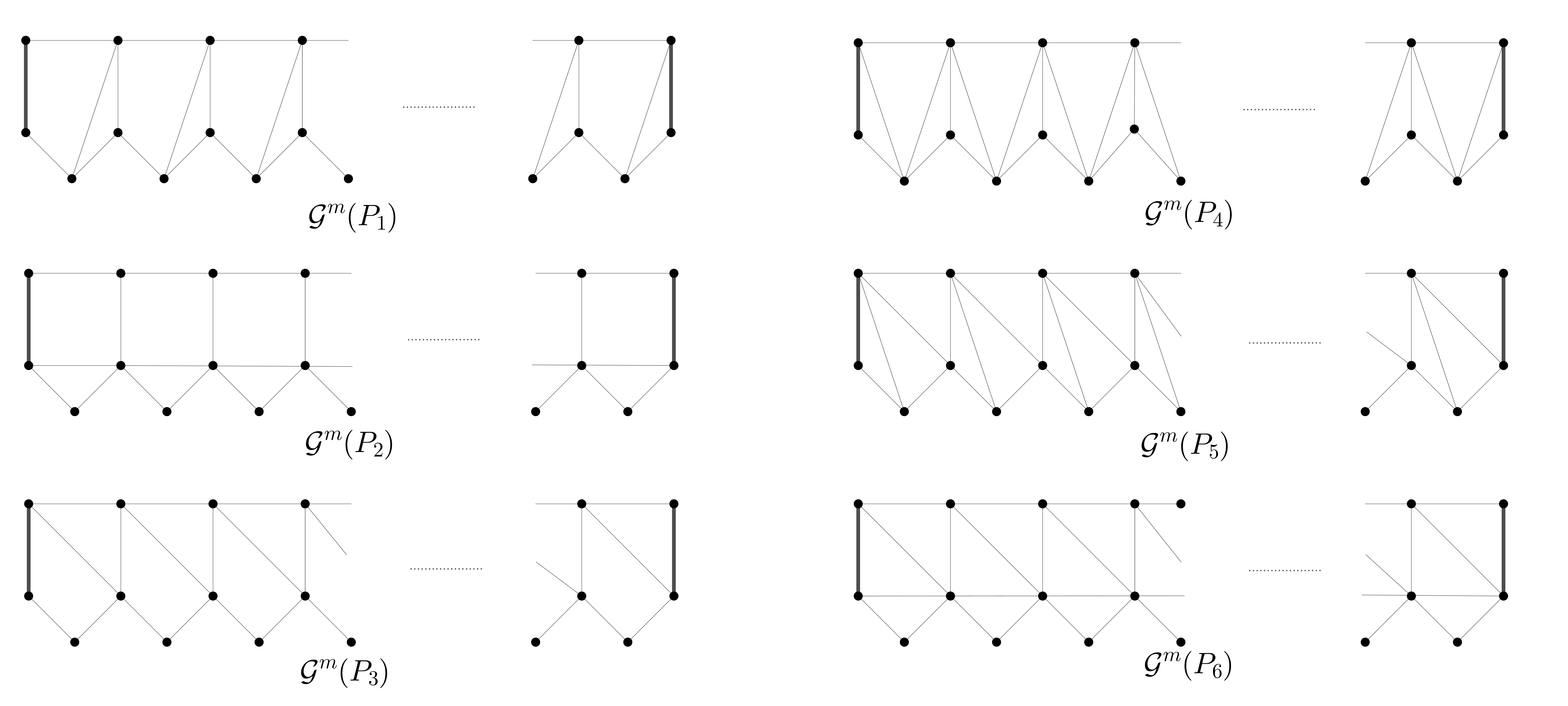}}\caption{Rotationally symmetric planar graphs arisen from pentagonal chorded cycles.}\label{Figure_PG5}
\end{center}
\end{figure}

\begin{figure}[ht]
\begin{center}
{\includegraphics[scale=0.3]{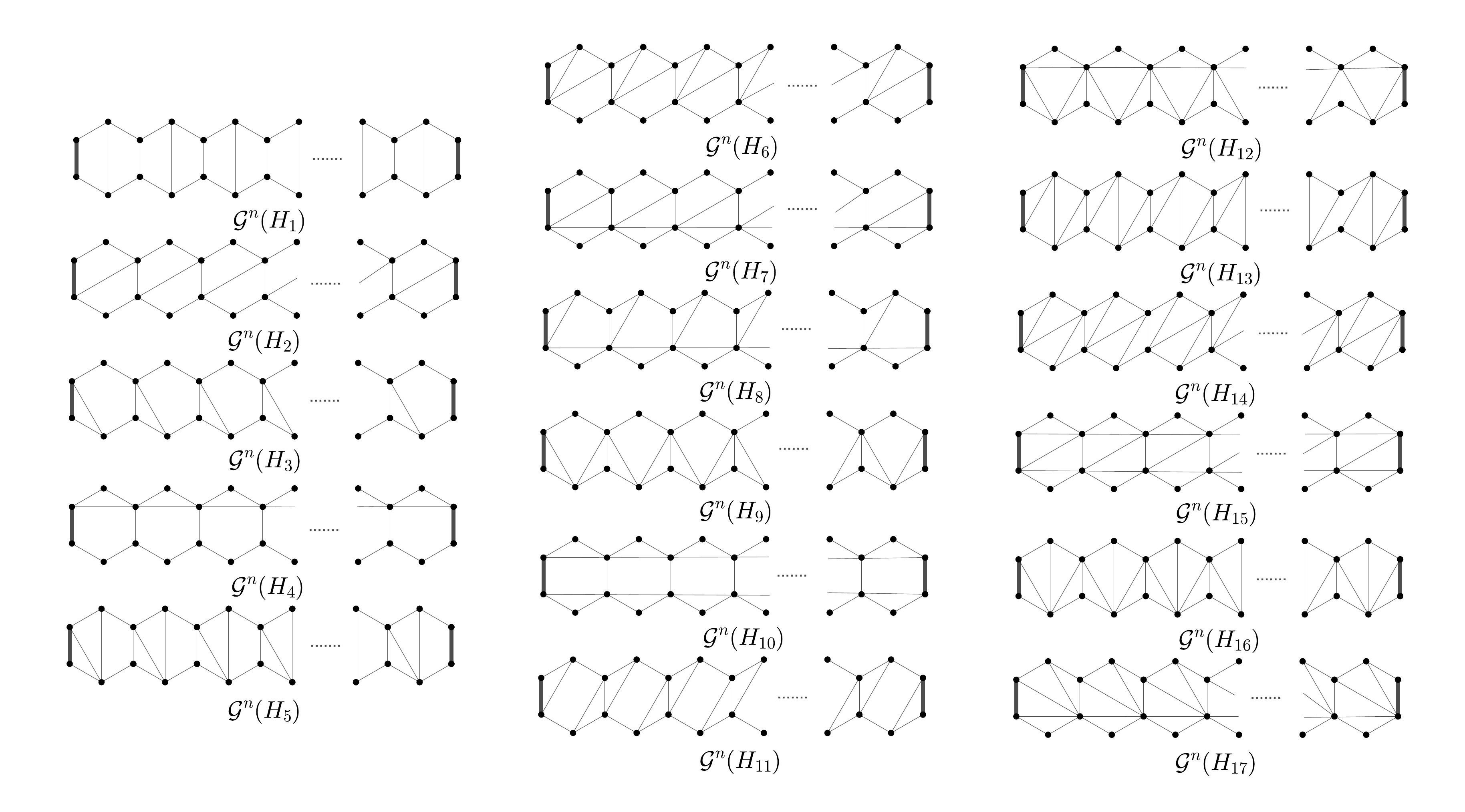}}\caption{Rotationally symmetric planar graphs arisen from hexagonal chorded cycles.}\label{Figure_RSPG}
\end{center}
\end{figure}

Table \ref{Table1} illustrates the local fractional metric dimension of each one of the planar chorded cycles described in Figure \ref{Figure_PG}. Notice in particular that the minimum value is reached for both graphs $H_1$ and $H_2$, which constitute different vertex-labelings of the same bipartite graph (see Lemma \ref{LemmaBenish}). In order to determine the remaining values, we have computationally solved the linear programming problem associated to each case \cite{Currie2001}.

\begin{example}\label{ExampleH5} The linear programming problem based on the local fractional metric dimension of the hexagonal chorded cycle $H_5$ consists of finding a map $\vartheta:V(H_5)\rightarrow [0,1]$ that minimizes $\sum_{1\leq i\leq 6} \vartheta(v_i)$ subject to
\[\begin{array}{ll}
\begin{cases}
\sum_{1\leq i\leq 6} \vartheta(v_i)\geq 1,\\
\vartheta(v_1)+\vartheta(v_2)+\vartheta(v_3)+\vartheta(v_4)+\vartheta(v_6)\geq 1,\\
\vartheta(v_3)+\vartheta(v_4)+\vartheta(v_5)+\vartheta(v_6)\geq 1,\\
\vartheta(v_1)+\vartheta(v_2)+\vartheta(v_5)\geq 1.
\end{cases}
\end{array}\]
Since $\mathcal{R}\{v_1,v_2\}=\mathcal{R}\{v_2,v_3\}=\mathcal{R}\{v_3,v_4\}=V(H_5)$, these three edges are related to the constraint  $\sum_{1\leq i\leq 6} \vartheta(v_i)\geq 1$. The second constraint follows from $\mathcal{R}\{v_1,v_4\}$ and $\mathcal{R}\{v_1,v_6\}$. The third one follows from $\mathcal{R}\{v_4,v_5\}$ and $\mathcal{R}\{v_5,v_6\})$. Finally, the fourth one follows from $\mathcal{R}\{v_1,v_5\}$. An optimal solution of this problem is the resolving function $\vartheta$ of $H_5$ satisfying that $\vartheta(v_i)=\frac 12$, if $i\in \{1,3,5\}$, and zero, otherwise. Hence, $\mathrm{ldim}_{\mathrm{f}}(H_5)=\frac 32$.
\end{example}

We are interested in the asymptotic behaviour of the local fractional metric dimension of each rotationally symmetric planar graph arisen from the planar chorded cycles described in Figure \ref{Figure_PG}. In this regard, we finish this section by dealing with those graphs arisen from quadrilateral chorded cycles. Those ones arisen from pentagonal and hexagonal chorded cycles are studied in Section \ref{sec:F1}. Firstly, we prove a preliminary lemma concerning the local fractional metric dimension of a wheel. (Notice that it differs from the unproven Theorem 3 in \cite{Aysiah2020}, whose third assertion seems not to be true.)

\begin{lemma}\label{LemmaW}  Let $W_n$ be the wheel graph of order $n\geq 4$. Then,
\[\mathrm{ldim}_{\mathrm{f}}(W_n)=\begin{cases}\begin{array}{cl}
2,&\text{ if } n=4,\\
\frac 32 ,&\text{ if } n\in\{5,6\},\\
\frac {n-1}4,& \text{ otherwise}.
\end{array}
\end{cases}\]
\end{lemma}

\begin{proof} The case $n\leq 6$ follows directly from solving the corresponding linear programming problem. In any case, notice that the case $n=4$ holds readily from Lemma \ref{LemmaBenish} once it is observed that every pair of vertices within the wheel graph $W_4$ are true twin.

Now, in order to deal with the case $n>6$, let us suppose that $V(W_n)=\left\{v_0,\ldots,v_{n-2},w\right\}$. Here, $w$ denotes the center of the wheel graph $W_n$. In addition, $v_iv_{i+1}\in E(W_n)$, for every non-negative integer $i<n$, where, from here on, all the indices are taken modulo $(n-1)$. Then, the computation of $\mathrm{ldim}_{\mathrm{f}}(W_n)$ requires to minimize $\vartheta(w)+\sum_{0\leq i\leq n-2} \vartheta(v_i)$ subject to
\[\begin{array}{ll}
\begin{cases}
\begin{array}{lr}
\begin{array}{l}
\vartheta(v_{i-1})+\vartheta(v_i)+\vartheta(v_{i+1})+\vartheta(v_{i+2})\geq 1,\\
\vartheta(w)-\vartheta(v_{i-1})-\vartheta(v_{i+1})+\sum_{0\leq i\leq n-2}\vartheta(v_i),
\end{array}& \forall i\leq n-2.
\end{array}
\end{cases}
\end{array}\]
The first constraint derives from  $\mathcal{R}\{v_i,v_{i+1}\}$, whereas the second one derives from $\mathcal{R}\{w,v_i\}$. An optimal solution of this problem is the resolving function $\vartheta$ of the wheel graph $W_n$ satisfying that $\vartheta(w)=0$ and $\vartheta(v_i)=\frac 14$, if $0\leq i\leq n-2$. Hence, $\mathrm{ldim}_{\mathrm{f}}(W_n)=\frac {n-1}4$.
\end{proof}

The following result establishes the local fractional metric dimension of $\mathcal{G}^m(Q_1)$ and $\mathcal{G}^m(Q_2)$.

\begin{proposition}\label{Proposition_GQ} Let $m\geq 2$ be a positive integer. Then,
\[\mathrm{ldim}_{\mathrm{f}}(\mathcal{G}^m(Q_1))=\begin{cases}\begin{array}{ll}
\frac 32,& \text{ if } m=2,\\
\frac m2,& \text{ otherwise}.
\end{array}
\end{cases} \hspace{0.5cm} \mathrm{ldim}_{\mathrm{f}}(\mathcal{G}^m(Q_2))=\begin{cases}\begin{array}{ll}
\frac 32,& \text{ if } m\leq 4,\\
\frac m4,& \text{ otherwise}.
\end{array}
\end{cases}\]
\end{proposition}

\begin{proof} For $Q_1$, the result holds from Lemma \ref{LemmaW}, because the planar graph $\mathcal{G}^m(Q_1)$ is a wheel graph of order $2m+1$. For $Q_2$, the case $m\in\{2,3,4\}$ follows directly from solving the corresponding linear programming problem. In order to deal with $m\geq 5$, the computation of $\mathrm{ldim}_{\mathrm{f}}(\mathcal{G}^m(Q_2))$ requires to solve the linear programming problem consisting on minimizing the objective function $\vartheta(v_3^1)+\sum_{1\leq i\leq m} (\vartheta(v_1^i)+\vartheta(v_4^i))$ so that the following constraints hold, for all $i\leq m$. Here, the superscripts are all of them taken modulo $m$.
    \begin{itemize}
        \item From $\mathcal{R}\{v_1^i,v_4^i\}$, it is $\vartheta(v_3^1)+\sum_{j\not\in\{i+1,\, i+2\}}\left(\vartheta(v_1^j)+\vartheta(v_4^j)\right)\geq 1$.

        \item From $\mathcal{R}\{v_3^1,v_4^i\}$, it is $\vartheta(v_3^1)+\sum_{j\not\in\{i-2,\, i+1\}}\vartheta(v_1^j)+\sum_{j\not\in\{i-1,\, i+1\}}\vartheta(v_4^j)\geq 1$.

         \item From $\mathcal{R}\{v_4^i,v_4^{i+1}\}$, it is $\vartheta(v_1^{i-2})+\vartheta(v_1^{i-1})+\vartheta(v_1^{i+1})+\vartheta(v_1^{i+2})+\vartheta(v_4^{i-1})+\vartheta(v_4^i)+\vartheta(v_4^{i+1})+\vartheta(v_4^{i+2})\geq 1$.
    \end{itemize}

An optimal solution of this problem is the resolving function $\vartheta$ of the wheel graph $W_n$ such that $\vartheta(v)=\frac 14$, if $v=v_4^i$, for some positive integer $i\leq m$, and zero, otherwise. Hence, $\mathrm{ldim}_{\mathrm{f}}(W_n)=\frac m4$.
\end{proof}

\section{Rotationally symmetric planar graphs based on pentagonal and hexagonal chorded cycles}\label{sec:F1}

This section studies the local fractional metric dimension problem of the rotationally symmetric planar graphs arisen from the pentagonal chorded cycles $P_1$ to $P_6$ and the hexagonal chorded cycles $H_1$ to $H_{17}$.

\subsection{The pentagonal case}

Except for $\mathcal{G}^m(P_1)$ and  $\mathcal{G}^m(P_5)$, this problem has recently been dealt with by Liu et al. \cite{Liu2020}, who have proved the following result.

\begin{proposition}[\cite{Liu2020}] \label{P2P5} Let $m\geq 4$ be an even positive integer. For each $i\in\{2,4,6\}$,
\[\mathrm{ldim}_{\mathrm{f}}(\mathcal{G}^m(P_i))\leq \frac {6m}{3m+2},\]
\end{proposition}

They have also considered the rotationally symmetric graph $\mathcal{G}^m(P_3)$, with $m\geq 4$ even, by indicating that $\mathrm{ldim}_{\mathrm{f}}(\mathcal{G}^m(P_3))\leq \frac {3m}{2m-1}$ (see \cite[Theorem 6]{Liu2020}). Nevertheless, such an upper bound is based on a wrong use of Condition (\ref{eq_Benish2}). More specifically, they indicate that $\left|\mathcal{R}\left\{v_4^i,\, v_5^i\right\}\right|=2m-1$, for every positive integer $i\leq m$ (see \cite[Lemma 6.(a)]{Liu2020}. However, as we notice in the proof of Proposition \ref{P0}, $\left|\mathcal{R}\left\{v_4^i,\, v_5^i\right\}\right|=\frac {3m+4}2$. A comprehensive study of the asymptotic behaviour of the local fractional metric dimension of all the graphs $\mathcal{G}^m(P_i)$, with $m\geq 2$ and $1\leq i\leq 6$, is therefore necessary.

\begin{proposition} \label{P0} Let $m\geq 2$ be a positive integer. Then,
\begin{enumerate}
    \item $\ell(\mathcal{G}^m(P_1))=\begin{cases}\begin{array}{ll}
m+2,& \text{ if } m\in\{2,3\},\\
    8,& \text{ otherwise}.
    \end{array}\end{cases}$

    \item $\ell(\mathcal{G}^m(P_2))=\frac {3m+3}2$.

    \item $\ell(\mathcal{G}^m(P_3))=\begin{cases}\begin{array}{ll}
3,& \text{ if } m=2,\\
5,& \text{ if } m=3,\\
\frac{3m+5}2, & \text{ if } m>3 \text{ is odd},\\
\frac{3m+4}2, & \text{ otherwise}.
    \end{array}\end{cases}$

    \item $\ell(\mathcal{G}^m(P_4))=\ell(\mathcal{G}^m(P_6))=\begin{cases}\begin{array}{ll}
5,& \text{ if } m=3,\\
\frac{3m+3}2, & \text{ otherwise}.
    \end{array}\end{cases}$

    \item $\ell(\mathcal{G}^m(P_5))=\begin{cases}\begin{array}{ll}
3,& \text{ if } m=2,\\
5,& \text{ if } m=3,\\
\frac{3m+3}2, & \text{ if } m>3 \text{ is odd},\\
\frac{3m+2}2, & \text{ otherwise}.
    \end{array}\end{cases}$
\end{enumerate}
\end{proposition}

\begin{proof} From the symmetry of the planar graphs under consideration, the result follows readily from the minimum cardinality of some of their resolving neighbourhoods. More specifically, it is enough to focus on those ones indicated in Tables \ref{Table0a} and \ref{Table0}, where, in order to simplify the notation, each vertex $v_a^bv_c^d$ is represented as $abcd$.
\end{proof}

\vspace{0.2cm}

Based on the previous result, the following theorem establishes upper bounds for the local fractional metric dimension of all the rotationally symmetric planar graphs arisen from the pentagonal chorded cycles $P_1$ to $P_6$.

\begin{theorem} \label{P1} Let $m\geq 2$ be a positive integer. Then,
\begin{enumerate}
    \item $\mathrm{ldim}_{\mathrm{f}}(\mathcal{G}^m(P_1))\leq \begin{cases}\begin{array}{ll}
\frac {3m}{m+2},& \text{ if } m\in\{2,3\},\\
\frac {3m}8,& \text{ otherwise}.
    \end{array}\end{cases}$.

    \item $\mathrm{ldim}_{\mathrm{f}}(\mathcal{G}^m(P_2))\leq \begin{cases}
\begin{array}{ll}
\frac {2m}{m+1}, & \text{ if } m \text{ is odd},\\
\frac {6m}{3m+2},  & \text{ otherwise}.
\end{array}
\end{cases}$

 \item $\mathrm{ldim}_{\mathrm{f}}(\mathcal{G}^m(P_3))\leq \begin{cases}
\begin{array}{ll}
2, & \text{ if } m=2,\\
\frac 95, & \text{ if } m=3,\\
\frac {6m}{3m+5}, & \text{ if } m>3 \text{ is odd},\\
\frac {6m}{3m+4},  & \text{ otherwise}.
\end{array}
\end{cases}$.

 \item If $i\in\{4,5,6\}$, then $\mathrm{ldim}_{\mathrm{f}}(\mathcal{G}^m(P_i))\leq \begin{cases}
\begin{array}{ll}
2,& \text{ if } (i,m)=(5,2),\\
\frac 95,&\text{ if } m=3,\\
\frac {2m}{m+1}, & \text{ if } m>3 \text{ is odd},\\
\frac {6m}{3m+2},  & \text{ otherwise}.
\end{array}
\end{cases}$
\end{enumerate}
\end{theorem}

\begin{proof}
For the planar graphs $P_2$, $P_4$ an $P_6$, the case $m$ even follows from Proposition \ref{P2P5} together with the computational resolution of the linear programming problem related to the case $m=2$, which gives rise to the values $\mathrm{ldim}_{\mathrm{f}}(\mathcal{G}^2(P_2))=\mathrm{ldim}_{\mathrm{f}}(\mathcal{G}^2(P_4))=\frac 32$ and $\mathrm{ldim}_{\mathrm{f}}(\mathcal{G}^2(P_6))=2$. The remaining cases follows all of them straightforwardly from Proposition \ref{P0} and Condition (\ref{eq_Benish2}).
\end{proof}

\vspace{0.1cm}

\subsection{The hexagonal case}

Similarly to the pentagonal case, the following preliminary technical result is required.

\begin{proposition}\label{0} Let $m\geq 2$ be a positive integer. Then,
\begin{enumerate}
    \item $\ell(\mathcal{G}^m(H_3))=2m+1$.
    \item If $i\in\{4,10,12,17\}$, then $\ell(\mathcal{G}^m(H_i))=\begin{cases}\begin{array}{ll}
4,& \text{ if } m=2,\\
2m+2,& \text{ otherwise}.
\end{array}
\end{cases}$
    \item If $i\in\{5,11,13\}$, then $\ell(\mathcal{G}^m(H_i))=2m$.

    \vspace{0.1cm}

    \item If $i\in\{6,9,14,16\}$, then $\ell(\mathcal{G}^m(H_i))=4$.

    \item $\ell(\mathcal{G}^m(H_7))=\begin{cases}
\begin{array}{ll}
4m-6,&\text{ if } m\leq 4,\\
2m+3,&\text{ otherwise}.
\end{array}
\end{cases}$

    \item $\ell(\mathcal{G}^m(H_8))=\begin{cases}
\begin{array}{ll}
2m,&\text{ if } m\leq 4,\\
8,&\text{ otherwise}.
\end{array}
\end{cases}$

    \item $\ell(\mathcal{G}^m(H_{15}))=\begin{cases}
\begin{array}{ll}
4m-6,&\text{ if } m\leq 3,\\
2m+2,&\text{ otherwise}.
\end{array}
\end{cases}$
\end{enumerate}
\end{proposition}

\begin{proof} Again, from the symmetry of the planar graphs under consideration, the result follows readily from the resolving neighbourhoods indicated in Tables \ref{Table6a}--\ref{Table6d}, together with the following consideration. (Again, each vertex $v_a^bv_c^d$ is represented as $abcd$ in the mentioned tables.) Under the following assumptions, it is verified that $|\mathcal{R}\{v,w\}|\geq 4$ for all $vw\in \left(E(\mathcal{G}^m(H_6))\setminus\{v_4^1v_6^1\}\right) \cup \left(E(\mathcal{G}^m(H_9))\setminus\{v_5^1v_6^1\}\right)\cup \left(E(\mathcal{G}^m(H_{14}))\setminus\{v_4^1v_6^1\}\right)\cup\left(E(\mathcal{G}^m(H_{16}))\setminus\{v_1^1v_4^1\}\right)$.
\end{proof}

\vspace{0.2cm}

Based on the previous result, the following theorem establishes upper bounds for the local fractional metric dimension of all the rotationally symmetric planar graphs arisen from the hexagonal chorded cycles $H_1$ to $H_6$.

\begin{theorem}\label{1} Let $m\geq 2$ be a positive integer. Then,
\begin{enumerate}
    \item $\mathrm{ldim}_{\mathrm{f}}(\mathcal{G}^m(H_1))=\mathrm{ldim}_{\mathrm{f}}(\mathcal{G}^m(H_2))=1$.

    \vspace{0.1cm}

    \item $\mathrm{ldim}_{\mathrm{f}}(\mathcal{G}^m(H_3))\leq \frac{4m}{2m+1}$.

    \vspace{0.1cm}

    \item If $i\in\{4,10,12,17\}$, then $\mathrm{ldim}_{\mathrm{f}}(\mathcal{G}^m(H_4))\leq\begin{cases}
\begin{array}{ll}
2,&\text{ if } m=2,\\
\frac{2m}{m+1},&\text{ otherwise}.
\end{array}
\end{cases}$

    \vspace{0.1cm}

    \item  If $i\in\{5,11,13\}$, then $\mathrm{ldim}_{\mathrm{f}}(\mathcal{G}^m(H_i))\leq 2$.

    \vspace{0.2cm}

    \item If $i\in\{6,9,14,16\}$, then $\mathrm{ldim}_{\mathrm{f}}(\mathcal{G}^m(H_i))\leq m$.

    \vspace{0.1cm}

    \item $\mathrm{ldim}_{\mathrm{f}}(\mathcal{G}^m(H_7))\leq \begin{cases}
\begin{array}{ll}
\frac{2m}{2m-3}, & \text{ if } m\leq 4,\\
\frac{4m}{2m+3}, & \text{ otherwise}.
\end{array}
\end{cases}$

    \vspace{0.1cm}

    \item $\mathrm{ldim}_{\mathrm{f}}(\mathcal{G}^m(H_8))\leq \begin{cases}\begin{array}{ll}
2,&\text{ if } m\leq 4,\\
\frac m2,&\text{ otherwise}.
\end{array}
\end{cases}$

    \vspace{0.1cm}

    \item $\mathrm{ldim}_{\mathrm{f}}(\mathcal{G}^m(H_{15}))\leq \begin{cases}\begin{array}{ll}
\frac {2m}{2m-3},&\text{ if } m\leq 3,\\
\frac{2m}{m+1},&\text{ otherwise}.
\end{array}
\end{cases}$
\end{enumerate}
\end{theorem}

\begin{proof} For the planar graphs $H_1$ and $H_2$, the result follows from Lemma \ref{LemmaBenish}, because both graphs $\mathcal{G}^m(H_1)$ and $\mathcal{G}^m(H_2)$ are bipartite. The remaining cases follows all of them straightforwardly from Proposition \ref{0} and Condition (\ref{eq_Benish2}).
\end{proof}

\section{Conclusion and further work}

In this paper, we have described a new family $\mathcal{G}^m(G)$ of rotationally symmetric planar graphs arisen from an edge coalescence of $m$ disjoint copies of a given planar chorded cycle $G$ of order $n$. For $n\leq 6$, we have obtained either the exact value or an upper bound of the local fractional metric dimension $\mathrm{ldim}_{\mathrm{f}}(\mathcal{G}^m(G))$, whatever the planar chorded cycle $G$ is. The obtained results are summarized in Table \ref{t1} just after the bibliography. Where possible, we indicate the asymptotic behaviour of the corresponding local fractional dimension. Particularly, our results generalize those ones obtained by Liu et al. \cite{Liu2020}, by considering the odd case for $n=5$, together with the case $m=2$.

Further work is still required in order to determine not only the asymptotic behaviour of the remaining cases, but also to deal with higher orders. In this regard, and according to the lower bound described in Condition (\ref{eq_Benish2}), the study of the local metric dimension of the rotationally symmetric planar graphs $\mathcal{G}^m(G)$ may play a relevant role. Similarly, the establishment of new general lower bounds concerning the local fractional metric dimension of any graph also constitutes a primordial aspect to be considered as further work. Finally, to delve into the study of structural properties of the rotationally symmetric planar graphs $\mathcal{G}^m(G)$ may also be of usefulness. In any case, let us remark that the local fractional dimension problem is still in a very initial stage, and much more work must be done even for establishing either exact values or lower/upper bounds for the most commonly used types of graphs. Lemma \ref{LemmaW} contributes in this regard by establishing the local fractional metric dimension of a wheel graph, which constitutes indeed a correction to the unproven Theorem 3 in \cite{Aysiah2020}.

\bibliographystyle{acm}
\bibliography{References.bib}

\section*{Appendix}

\appendix

\section{Tables}\label{app}

\begin{table}[ht]
\centering
\begin{tabular}{c|c}
G & $\mathrm{ldim}_{\mathrm{f}}(G)$ \\ \hline
$Q_i$ ($1\leq i\leq 2$) & $3/2$\\
$P_i$ ($1\leq i\leq 6$) & $3/2$\\
$H_i$  ($1\leq i\leq 2$) & $1$\\
$H_i$  ($3\leq i\leq 12$) & $3/2$\\
$H_i$  ($13\leq i\leq 17$) & $5/3$
    \end{tabular}
        \vspace{0.1cm}
    \caption{Local fractional dimension of planar chorded cycles of order $n\leq 6$.}
    \label{Table1}
\end{table}

\renewcommand{\tabcolsep}{2pt}
\begin{table}[ht]
 \centering
\begin{tabular}{llll}
     $n$ & $abcd$   &  $\overline{\mathcal{R}}\{v_a^b,v_c^d\}$ & Constraint\\ \hline
$1$ & $1142$  & \multirow{ 2}{*}{$\left\{12,\, 52,\, 1m,\, 5m\right\}$} & For all $m$\\
     & $4151$ & & For all $m$\\
     & $1151$ & $\left\{42\right\}$ & $m=3$\\
     & & $\left\{1\frac{m+1}2,\, 4\frac {m+3}2,\, 5\frac{m+3}2\right\}$ & $m>3$ odd\\
     & & $\emptyset$ & Otherwise\\
     & $1152$ & $\left\{4i\colon\,1\leq i\leq m\right\}$ & $m\in\{2,3\}$\\
     & & $ V(\mathcal{G}^m(P_1))\setminus\left\{1i,\,5i\colon\, i\in\{1,2,3,m\}\,\right\}$ & Otherwise\\
     & $4142$ & $\left\{1\frac{m+1}2,\, 4\frac{m+3}2,\, 5\frac{m+3}2\right\}$ & $m$ odd\\
     & & $\emptyset$ & Otherwise\\
$2$ & $1151$ & $V(\mathcal{G}^m(P_2))\setminus\left\{11,\, 41,\,51,\, 1i,\, 4i,\, 5i\colon\, \frac{m+1}2<i\leq m\right\}$ & For all $m$\\
     & $4142$ & \multirow{ 2}{*}{$\left\{11,\,4\frac {m+3}2,\, 5\frac {m+3}2\right\}$} & For all $m$\\
     & $5152$ & &  For all $m$ \\
     & $4151$ & $\emptyset$ &  For all $m$\\
    \end{tabular}

    \vspace{0.1cm}
 \caption{Resolving neighbourhoods of $\mathcal{G}^m(P_n)$, for $n\in\{1,2\}$.}\label{Table0a}
    \end{table}

\renewcommand{\tabcolsep}{1pt}
\begin{table}[ht]
 \centering
\begin{tabular}{llll}
     $n$ & $abcd$   &  $\overline{\mathcal{R}}\{v_a^b,v_c^d\}$ & Constraint\\ \hline
$3$ &  $1151$   & $\left\{42\right\}$ & $m\in\{2,3\}$\\
    &  & $\left\{42,\, 1i,\,4i,\, 5i\colon\, 3\leq i\leq \frac{m+1}2\right\}$ & $m>3$ odd\\
     & & $\left\{42,\,1i,\,4j,\, 5j\colon\, 3\leq i\leq \frac m2,\, 3\leq j\leq \frac {m+3}2\right\}$ & Otherwise\\
   & $1152$ &   $\emptyset$ & $m=2$\\
   & & $\left\{43\right\}$ & $m=3$\\
  &  & $\left\{4\frac {m+3}2,\,1i,\,4(i+1),\, 5(i+1)\colon\, \frac{m+3}2\leq i<m\right\}$ & $m>3$ odd\\
   & & $\left\{5(\frac m2+2)\right\}$ & Otherwise\\
   & $4142$ & $\left\{11,\, 12,\,52\right\}$ & $m=2$\\
   & & $\left\{11,\, 12,\,43,\,52\right\}$ & $m=3$\\
   & & $ \left\{11,\,12,\,1\frac{m+3}2,\,4\frac{m+3}2,\,52\right\}$ & $m>3$ odd\\
   & & $ \left\{11,\,12,\,4(\frac m2 +2),\,52\right\}$ & Otherwise\\
   & $4151$ & $\emptyset$ & $m=2$\\
   & & $\left\{42,\,52\right\}$ & $m=3$\\
   & & $\left\{4m,\, 5m,\, 1i,\,4i,\,5i\colon\, \frac{m+3}2\leq i<m\right\}$ & $m>3$ odd\\
    & & $\left\{4\frac m2,\,1i,\,4(i+1),\,5(i+1)\colon\, \frac m2<i<m\right\}$ & Otherwise\\
    & $4152$ & $\left\{42,\,53\right\}$ & $m=3$\\
    & & $\left\{1i,\, 4j,\,5(j+1)\colon\, 3\leq i\leq \frac{m+1}2,\, 2\leq j\leq \frac {m+1}2\right\}$ & $m>3$ odd\\
    & & $\left\{42,\,1i,\, 4i,\, 5i\colon\,3\leq i\leq \frac {m+3}2\right\}$ & Otherwise\\
$4$ & $1141$ & $\left\{51,\,1i,\, 5j,\, 4i\colon\, 2\leq i\leq \frac{m+1}2,\,3\leq j\leq \frac {m+1}2\right\}$ & For all $m$\\
    & $1151$ &  $\left\{41,\,1i,\, 4j,\, 5i\colon\, \frac{m+3}2\leq i<m,\, \frac {m+3}2\leq j\leq m\right\}$ & For all $m$\\
    & $4142$ & $\left\{4\frac{m+3}2,\,5\frac{m+3}2\right\}$ & For all $m$\\
    & $4151$ & $\left\{11,\,1m,\, 52,\,5m\right\}$ & For all $m$\\
$5$ & $1141$ & $\left\{12,\,1m,\,51,\,52\right\}$ & For all $m$\\
    & $1151$ & $\left\{41\right\}$ & $m=2$\\
    & & $\left\{1i,\,41,\,42,\, 4i,\,5i\colon\, 3\leq i\leq \frac {m+1}2\right\}$ & $m$ odd\\
    & & $\left\{1i,\,41,\,42,\, 4i,\,5(\frac {m+3}2),\,5i\colon\, 3\leq i\leq \frac m2\right\}$ & Otherwise\\
    & $1152$ & $\begin{array}{l}\left\{41,\,1i,\,4j,\,5k\colon\,\frac{m+3}2\leq i<m,\,\frac{m+3}2\leq j\leq m\,\frac{m+3}2< k\leq m\right\}\end{array}$ & $m$ odd\\
    & & $\left\{41,\,1i,\,4j,\, 5j\colon\,\frac m2+2\leq i<m,\,\frac m2+2\leq j\leq m\right\}$ & $m$ even\\
    & $4142$ & $\left\{51\right\}$ & $m=2$\\
    & & $\left\{1\frac {m+3}2,\,4\frac{m+3}2,\, 51\right\}$ & $m$ odd\\
    & & $\left\{51,\, 5(\frac m2+2\right\}$ & Otherwise\\
    & $4151$ & $\begin{array}{l}V(\mathcal{G}^m(P_5))\setminus \left\{11,\,1i,\,4j,\,5j\colon\,\frac{m+3}2\leq i<m,\,\frac{m+3}2\leq j\leq m\right\}\end{array}$ & $m$ odd\\
    & & $\begin{array}{l}V(\mathcal{G}^m(P_5))\setminus \left\{11,\,1i,\,4(\frac {m+3}2),\,4j,\, 5j\colon\,\frac m2<i<m,\,\frac {m+3}2<j\leq m\right\}\end{array}$ & Otherwise\\
    & $4152$ & $\left\{11,\,42,\,1i,\,4i,\,5i,\,5\frac {m+3}2\colon\,3\leq i\leq \frac{m+1}2\right\}$ & $m$ odd\\
    & & $\left\{11,\,42,\,1i,\,4i,\,5i\colon\,3\leq i\leq \frac {m+3}2\right\}$ & Otherwise\\
$6$ & $1151$ & $\left\{1i,\, 4i,\, 5i\colon\, 2\leq i\leq \frac{m+1}2\right\}$ & For all $m$\\
    & $1152$ & $\left\{1i,\, 4i,\,51,\, 5j\colon\, \frac{m+3}2\leq i\leq m,\,\frac{m+3}2<j\leq m\right\}$ & For all $m$\\
    & $4142$ & $\left\{11,\, 12,\, 43,\, 52\right\}$ & $m=3$\\
    & & $\left\{11,\,12,\,1\frac{m+3}2,\, 4\frac{m+3}2,\, 52\right\}$ & Otherwise\\
    & $4151$ & $\left\{1i,\, 4j,\, 5i\colon\, 2\leq i\leq\frac{m+1}2,\, \frac{m+3}2\leq j\leq m\right\}$ & For all $m$\\
    & $4152$ & $\left\{1i,\, 4j,\,51,\, 5k\colon\, \frac{m+3}2\leq i\leq m,\, \frac{m+3}2<j\leq m,\, 2\leq k\leq \frac{m+1}2 \right\}$ & For all $m$\\
    & $5152$ & $\left\{11,\, 41,\, 5{\frac{m+3}2}\right\}$ & For all $m$
    \end{tabular}
    \vspace{0.1cm}
 \caption{Resolving neighbourhoods of $\mathcal{G}^m(P_n)$, for $3\leq n\leq 6$.}\label{Table0}
    \end{table}

\begin{table}[ht]
 \centering
\begin{tabular}{llll}
     $n$ & $abcd$   &  $\overline{\mathcal{R}}\{v_a^b,v_c^d\}$ & Constraint\\ \hline
$3$ & $1151$ & $\left\{61,\, 1i,\,4j,\,5j,\,6i\colon\, 2\leq j\leq \frac {m+1}2<i\leq m\right\}$ & $m$ odd\\
     & & $\left\{61,\, 1i,4j,\,5k,\,\,6l\colon\, 2\leq j\leq \frac m2<i\leq m,\,2\leq k\leq \frac {m+3}2<l\leq m \right\}$ & Otherwise\\
    & $1161$ & $\left\{41,\,4m,\,51\right\}$ & For all $m$\\
    & $1162$ & $\left\{41,5{\frac {m+1}2}\right\}$ & $m$ odd\\
    & & $\left\{41,4({\frac m2 +1})\right\}$ & Otherwise\\
    & $4151$ & $\{62\}$ & $m$ odd\\
    & & $\left\{1({\frac {m+3}2}),\, 62\right\}$ & Otherwise\\
    & $4152$ & $\left\{62,6{\frac {m+3}2}\right\}$ & $m$ odd\\
    & & $ \left\{62,\,1({\frac {m+3}2})\right\}$ & Otherwise\\
    & $5161$ & $\left\{1i,\,4j,\,5k,\,6l\colon\, 1\leq i\leq \frac {m+1}2\leq j<m,\, 2\leq l\leq \frac {m+1}2<k\leq m\right\}$ & $m$ odd\\
    & & $\left\{1i,\,4j,\,5k,\, 6l\colon\, 2<i\leq \frac m2<j<m,\, 2\leq l\leq \frac {m+3}2,\,\frac m2<k\leq m\right\}$ & Otherwise\\
    $4$ & $1161$ & $\left\{4i,\,5i\colon\,2\leq i\leq \frac{m+1}2\right\}$ & $m$ odd\\
    & & $\left\{4i,\,5j\colon\,2\leq i\leq \frac m2,\, 2\leq j\leq \frac {m+3}2\right\}$ & Otherwise\\
    & $4151$ & $\left\{ij\colon\, i\in\{1,4,5,6\},\, 2\leq j\leq \frac {m+1}2\right\}$ & $m$ odd\\
    & & $\left\{1i,\,4i,\,5j,\,6j\colon\, 2\leq i\leq \frac m2,\, 2\leq j\leq \frac {m+3}2\right\}$ & Otherwise\\
    & $5152$ & $\left\{11,\,41,5{\frac{m+1}2},\,6{\frac{m+1}2}\right\}$ & $m$ odd\\
    & & $\left\{11,\,41,1({\frac m2 + 1}),\,4({\frac {m+3}2})\right\}$ & Otherwise\\
    & $5161$ & $\{62,\,6m\}$ & $m=2$\\
    & & $\{12,\,1(m-1),\,62,\,6m\}$ & Otherwise\\
$5$ & $1141$ & $\left\{51,\, 1i,\, 4j,\,5j,\,6k\colon\, 2\leq i\leq \frac {m+1}2<j\leq m,\,  2<k\leq \frac {m+3}2\right\}$ & $m$ odd\\
    & & $\left\{51,\,1i,\,4j,\,5k,\,6l\colon\, 2\leq i\leq \frac {m+3}2,\, \frac m2<j\leq m,\, 2<l\leq \frac {m+3}2<k\leq m\right\}$ & Otherwise\\
 & $1151$ & $\left\{61,\,1i,\,4j,\,5k,\, 6i\colon\, 1\leq j\leq \frac {m+1}2<i\leq m,\,2\leq k\leq \frac {m+1}2\right\}$ & $m$ odd\\
& & $\left\{61,\, 1i,\,4j,\,5k,\, 6l\colon\, 2\leq k\leq \frac {m+3}2\leq i\leq m,\,1\leq j\leq \frac m2,\,  \frac {m+3}2<l\leq m\right\}$ & Otherwise\\
& $1161$ & $\left\{4m,\, 51\right\}$ & $m$ odd\\
& & $\left\{4m,\, 51,\, 5({\frac {m+3}2})\right\}$ & Otherwise\\
& $1162$ &  $\left\{5{\frac {m+3}2}\right\}$ & $m$ odd\\
& & $\left\{4({\frac {m+3}2})\right\}$ & Otherwise\\
& $4151$ & $\left\{11,\, 62,\, 6{\frac {m+3} 2}\right\}$ & $m$ odd\\
& & $\left\{11,\, 62, \, 1({\frac {m+3}2})\right\}$ & Otherwise\\
& $4152$ & $\left\{1{\frac {m+3}2}\right\}$ & $m$ odd\\
& & $\left\{6({\frac m2+2})\right\}$ & Otherwise\\
& $5161$ & $\left\{1i,\, 4j,\, 5k,\,6l,\colon\, 1\leq i\leq \frac {m+1}2< k\leq m ,\,  2\leq l\leq \frac {m+1}2\leq j<m\right\}$ & $m$ odd\\
& & $\left\{1i,\,4j,\,5k,\,6l\colon\, 1\leq i\leq \frac m2<k\leq m,\,  2\leq l\leq \frac {m+3}2\leq j<m\right\}$ & Otherwise\\
$6$ & $4161$ & $V(\mathcal{G}^m(H_6))\setminus \left\{11,\,1m,\,41,\,61\right\}$ & For all $m$\\
\end{tabular}
\vspace{0.1cm}
 \caption{Resolving neighbourhoods of $\mathcal{G}^m(H_n)$, for $3\leq n\leq 6$.}\label{Table6a}
\end{table}

\renewcommand{\tabcolsep}{1pt}
\begin{table}[ht]
 \centering
\begin{tabular}{llll}
     $n$ & $abcd$   &  $\overline{\mathcal{R}}\{v_a^b,v_c^d\}$ & Constraint\\ \hline
    $7$ & $1161$ & $\begin{array}{l}
\left\{1i,\, 6i,\, 4j,\,5k\colon\, 2\leq i\leq \frac {m+1}2,\, 2< j\leq \frac {m+1}2+1,\, 2< k\leq \frac {m+1}2\right\}\end{array}$ & $m$ odd\\
& & $\begin{array}{ll}
\left\{1i,\, 6j,\,4k,\, 5k\colon\, 2\leq i\leq \frac m2,\, 2\leq j\leq \frac {m+3}2,\, 2<k\leq \frac {m+3}2\right\}\end{array}$ & Otherwise\\
& $1162$ & $\left\{51,\,61,\, 1i,\,4j,\,5j,\,6j\colon\, \frac {m+1}2< i\leq m,\,\frac {m+3}2<j\leq m\right\}$ & $m$ odd\\
& & $\left\{51,\,61,\,1i,\, 4i, 6i,\, 5j\colon\, \frac {m+3}2<i\leq m,\, \frac m2+2<i\leq m\right\}$ & Otherwise\\
& $4151$ & $\left\{1i,\,6j,\,4k,\,5k\colon\, 2\leq i< \frac {m+1}2,\,2\leq j\leq \frac {m+1}2,\, 2<k\leq \frac {m+1}2\right\}$ & $m$ odd\\
& & $\left\{12,\,53,\,62\right\}$ & $m=4$\\
& & $\left\{1i,\,4j,5j, 6i\colon\, 2\leq i\leq \frac m2,\, 2<j\leq \frac m2\right\}$ & Otherwise\\
& $4152$ & $\left\{1i,4i,5j,6k\colon\, \frac {m+1}2< i<m,\, \frac {m+3}2<j\leq m,\, \frac {m+1}2<k\leq m\right\}$ & $m$ odd\\
& & $\left\{1i,\,4j,5k,\,6k\colon\, \frac {m+3}2\leq i<m,\, \frac {m+3}2<j< m,\, \frac {m+3}2<k\leq m\right\}$ & Otherwise\\
& $5161$ & $\left\{1i,\,5j,\,6j,\, 4k\colon\, \frac {m+1}2\leq i<m,\, \frac {m+1}2< j\leq m,\,\frac {m+1}2<k<m\right\}$ & $m$ odd\\
& & $\left\{1i,\,4i,\,6j,\, 5k\colon\, \frac {m+3}2\leq i<m,\, \frac {m+3}2\leq j\leq m,\, \frac {m+3}2<k\leq m\right\}$ & Otherwise\\
& $5261$ & $\left\{1i,\,6i,\,4j,\,5k\colon\, 2\leq i\leq \frac {m+1}2,\, 2<j\leq \frac{m+1}2 ,\, 2<k\leq\frac {m+3}2\right\}$ & $m$ odd\\
& & $\left\{1i,6j,4k,5k\colon\, 2\leq i\leq \frac m2,\, 2\leq j\leq \frac {m+3}2,\,2<k\leq \frac {m+3}2\right\}$ & Otherwise\\
& $6162$ & $\left\{11,\,41,\,42,\,4{\frac {m+3}2},\,52,\,6{\frac {m+3}2}\right\}$ & $m$ odd\\
& & $\left\{11,1({\frac {m+3}2}),41,42,52,5({\frac m2+2})\right\}$ & Otherwise\\
$8$ & $1161$ & $\left\{1i,\,4i,\,5i,\, 6i\colon\, 2\leq i\leq \frac {m+1}2\right\}$ & $m$ odd\\
& & $\left\{1i,\, 4j,\, 5j,\,6j\colon\,2\leq i\leq \frac m2,\, 2\leq j\leq \frac {m+3}2\right\}$ & Otherwise\\
& $1162$ & $\begin{array}{ll}
\left\{41,\,51,\,61,\,1i,\,4j \colon\, \frac {m+1}2<i\leq m,\,\frac {m+3}2<j\leq m\right\}\cup\\
\cup\left\{5i,\,6i\colon\, \frac {m+1}2<i<m\right\}\end{array}$ & $m$ odd\\
& & $\left\{41,\,51,\,61,\,1i,\,4i,\,5i,\,6i\colon\, \frac {m+3}2<i\leq m\right\}$ & Otherwise\\
& $4151$ & $V(\mathcal{G}^m(H_8))\setminus\left\{41,\,42,\,51,\,52\right\}$ & $m=2$\\
&  & $V(\mathcal{G}^m(H_8))\setminus\left\{41,\,42,\,4(m-1),\,4m,\,51,\,52,\,53,\,5m\right\}$ & Otherwise\\
& $4152$ & $\left\{11,\,4{\frac {m+1}2}\right\}$ & $m$ odd\\
& & $\left\{11,\,1{\frac m2}\right\}$ & Otherwise\\
& $4161$ & $\left\{42,\,4m,\,51,\,53\right\}$ & For all $m$\\
& $5161$ & $\left\{41,\,4(m-1),\,52,\,5m\right\}$ & For all $m$\\
& $6162$ & $\left\{11,\,6{\frac {m+3}2}\right\}$ & $m$ odd\\
& & $\left\{11,\,1{\frac {m+3}2}\right\}$ & Otherwise\\
$9$ & $5161$ & $V(\mathcal{G}^m(H_9))\setminus\left\{41,\,4m,\,51,\,61\right\}$ & For all $m$\\
$10$ & $1161$ & $\left\{1i,\,4i,\,5i,\,6i\colon\, 2\leq i\leq \frac {m+1}2\right\}$ & $m$ odd\\
& & $\left\{1i,\,4i,\,5j,\,6j\colon\, 2\leq i\leq \frac m2,\, 2\leq j\leq \frac {m+3}2\right\}$ & Otherwise\\
& $5161$ & $\emptyset$ & For all $m$\\
& $6162$ & $\left\{11,\,41,\,5{\frac {m+3}2},6{\frac {m+3}2}\right\}$ & $m$ odd\\
& & $\left\{11,\,1({\frac {m+3}2}),\,41,\,4({\frac {m+3}2})\right\}$ & Otherwise\\
$11$ & $1161$ & $\left\{5{\frac {m+3}2}\right\}$ & $m$ odd\\
& & $\left\{4({\frac {m+3}2)}\right\}$ & Otherwise\\
& $4151$ & $\left\{11,\, 1{\frac {m+1}2},\, 61\right\}$ & $m$ odd\\
& & $\left\{11,\, 61,\, 6{\frac m2}\right\}$ & Otherwise\\
& $4161$ & $\left\{51,\,1i,\, 6i,\, 4j,\, 5j\colon\, 2\leq i\leq {\frac {m+1}2}<j\leq m\right\}$ & $m$ odd\\
& & $\left\{51,\, 1i,\, 6j,\, 4k,\, 5l\colon\, 2\leq i\leq \frac m2<k\leq m,\, 2\leq j\leq \frac {m+3}2<l\leq m\right\}$ & Otherwise\\
& $5161$ & $\left\{1i,\, 6j,\, 4k,\, 5l \colon\, 1\leq k\leq \frac {m+1}2\leq i\leq m,\, 2\leq l\leq \frac {m+1}2<i\leq m\right\}$ & $m$ odd\\
& & $\begin{array}{ll}
\left\{1i,\, 6i,\, 4j,\,5k\colon\, \frac m2<i\leq m,\, 1\leq j\leq \frac m2,\, 2\leq k\leq \frac {m+3}2\right\}
\end{array}$ & Otherwise\\
\end{tabular}
\vspace{0.1cm}
 \caption{Resolving neighbourhoods of $\mathcal{G}^m(H_n)$, for $7\leq n\leq 11$.}\label{Table6b}
\end{table}

\begin{table}[ht]
 \centering
\begin{tabular}{llll}
     $n$ & $abcd$   &  $\overline{\mathcal{R}}\{v_a^b,v_c^d\}$ & Constraint\\ \hline
    $12$ & $1151$ & $\left\{61,\, 1i,\,4i,\, 5j,\, 6k\colon\, 2\leq i\leq {\frac {m+1}2},\, 2\leq j\leq {\frac {m+1}2},\,2< k\leq {\frac {m+1}2}\right\}$ & $m$ odd\\
& & $\left\{61,\,1i,\,4i,\,5j,\, 6k\colon\, 2\leq i\leq \frac m2,\, 2\leq j\leq \frac {m+3}2,\,2< k\leq \frac {m+3}2\right\}$ & Otherwise\\
& $1161$ & $\left\{41,\, 51,\, 1i,\,4j,\,5j,\,6k\colon \frac {m+1}2\leq i<m, \, \frac {m+1}2<j\leq m,\, \frac {m+1}2<k<m\right\}$ & $m$ odd\\
& & $\begin{array}{ll}
\left\{41,\,51,\,1i,\,6j\colon\frac m2<i<m, \frac {m+3}2<j<m\right\}\cup\\
\cup\left\{4i,\,5j\colon \frac m2<i\leq m, \frac m2+2\leq j\leq m\right\}
\end{array}$ & Otherwise\\
& $4151$ & $\left\{1i,\,4i,\, 5i,\, 6i\colon\, 2\leq i\leq \frac {m+1}2\leq i<m\right\}$ & $m$ odd\\
& & $\left\{1i,\,4i,\, 5j,\, 6j\colon\,2\leq i\leq \frac m2,\, 2\leq j\leq \frac {m+3}2\right\}$ & Otherwise\\
& $5152$ & $\begin{array}{ll}
\left\{11,\,41,\, 5{\frac {m+3}2},\, 6{\frac {m+3}2}\right\}\end{array}$ & $m$ odd\\
& & $\left\{11,\,1({\frac {m+3}2}),\, 41,\, 4({\frac {m+3}2})\right\}$ & Otherwise\\
& $5161$ & $\left\{11,\,1m,\, 62,\, 6m\right\}$ & For all $m$\\
$13$ & $1141$ & $\left\{61,\,1i,\,4j,\, 5k,\, 6i\colon\, 2\leq j\leq \frac {m+1}2<i\leq m,\,2\leq k\leq {\frac {m+3}2}\right\}$ & $m$ odd\\
& & $\left\{61,\,1i,\,4j,\,5j,\, 6k\colon\, \frac m2<i\leq m,\, 2\leq j\leq\frac {m+3}2<j\leq m\right\}$ & Otherwise\\
& $1161$ & $\left\{41,\, 5{\frac {m+3}2}\right\}$ & $m$ odd\\
& & $\left\{41,\,4({\frac {m+3}2)}\right\}$ & Otherwise\\
& $4151$ & $\left\{61,\, 2{\frac {m+1}2}\right\}$ & $m$ odd\\
& & $\left\{61,\,6({\frac {m+3}2})\right\}$ & Otherwise\\
& $4161$ & $\left\{51,\, 1i,\,4j,\, 5j,\, 6k\colon 1\leq i\leq \frac {m+1}2<j\leq m,\, 2\leq k\leq \frac {m+1}2\right\}$ & $m$ odd\\
& & $\left\{51,\, 1i,\,4j,\,5k,\, 6l\colon\, 1\leq i\leq \frac m2<j\leq m,\, 2\leq l\leq \frac {m+3}2<k\leq m\right\}$ & Otherwise\\
& $5161$ & $\left\{1i,\,4j,\,5k,\,6l\colon  1\leq j\leq \frac {m+1}2\leq i\leq m,\, 2\leq k\leq \frac {m+1}2<l\leq m\right\}$ & $m$ odd\\
& & $\left\{1i,\,4j,\,5k,\, 6i\colon\, \frac {m+3}2\leq i\leq m,\,1\leq j\leq \frac m2,\, 2\leq k\leq \frac {m+3}2\right\}$ & Otherwise\\
$14$ & $4161$ & $V(\mathcal{G}^m(H_{14}))\setminus\left\{11,\,1m,\,41,\,61\right\}$ & For all $m$\\
$15$ & $1161$ &  $\left\{1i,\,4j,\, 5i,\,6i\colon 2\leq i\leq \frac {m+1}2,\, 2< j\leq \frac {m+1}2\right\}$ & $m$ odd\\
& & $\left\{1i,\,4j,\,5j,\,6k\colon\, 2\leq i\leq \frac m2,\,2<j\leq \frac {m+3}2,\, 2\leq k\leq \frac {m+3}2\right\}$ & Otherwise\\
& $5152$ & $\left\{11,\, 1{\frac {m+1}2},\, 1m,\, 41,\, 5{\frac {m+3}2},\, 61\right\}$ & $m$ odd\\
& & $\left\{11,\, 1m,\,41,\, 4({\frac {m+3}2}),\,61,\, 6({\frac {m+3}2})\right\}$ & Otherwise\\
& $5161$ & $\left\{1i,\,4j,\, 5j,\,6k\colon \frac {m+1}2\leq i<m,\, 2\leq j\leq \frac {m+1}2<k\leq m\right\}$ & $m$ odd\\
& & $\left\{1i,\,4j,\,5k,\,6l\colon\, 2\leq j\leq\frac m2<i<m,\,  2\leq k\leq \frac {m+3}2\leq l\leq m\right\}$ & Otherwise\\
& $6152$ & $\left\{51,\,1i,\,4j,\,5k,\,6i\colon 2\leq i\leq \frac {m+1}2<j\leq m,\frac{m+3}2<k\leq m\right\}$ & $m$ odd\\
& & $\begin{array}{ll}
\left\{51,\,1i,\,6j\colon 2\leq i\leq \frac m2, 2\leq j\leq\frac {m+3}2\right\}\\
\left\{4i,\,5j\colon \frac {m+3}2<i\leq m, \frac {m+3}2\leq j\leq m\right\}
\end{array}$ & Otherwise\\
$16$ & $1141$ & $V(\mathcal{G}^m(H_{14}))\setminus\left\{11,\,41,\,61,\,62\right\}$ & For all $m$\\
\end{tabular}
\vspace{0.1cm}
 \caption{Resolving neighbourhoods of $\mathcal{G}^m(H_n)$, for $12\leq n\leq 16$.}\label{Table6c}
\end{table}

\renewcommand{\tabcolsep}{1pt}
\begin{table}[ht]
 \centering
\begin{tabular}{llll}
     $abcd$   &  $\overline{\mathcal{R}}\{v_a^b,v_c^d\}$ & Constraint\\ \hline
$1161$ & $\left\{1i,\,4j,\,5i,\,6i\colon\, 2\leq i\leq \frac {m+1}2,\,1\leq i< \frac {m+1}2\right\}$ & $m$ odd\\
& $\left\{1i,\,4j,\,5k,\, 6l\colon\, 2\leq i\leq \frac m2,\,1\leq j\leq \frac m2,\, 2\leq k\leq \frac m2,\, 2\leq l\leq \frac {m+3}2\right\}$ & Otherwise\\
$1162$ & $\left\{61,\, 1i,\,4j,\,5i,\,6i\colon\, \frac {m+1}2<i\leq m,\, \frac {m+1}2<j<m\right\}$ & $m$ odd\\
& $\left\{61,\, 1i,\,4j,\,5k,\,6i\colon\, \frac {m+3}2<i\leq m,\,\frac {m+3}2\leq j<m,\, \frac {m+3}2\leq k\leq m\right\}$ & Otherwise\\
$4151$ & $\left\{1i,\,4j,\, 5j\,6k\colon\, 1\leq i\leq \frac {m+1}2,\, 3\leq j\leq  \frac {m+1}2,\, 2\leq k\leq \frac {m+3}2\right\}$ & $m$ odd\\
& $\begin{array}{ll}
\left\{1i,\,6j\colon\, 1\leq i\leq \frac {m+3}2,\, 2\leq j\leq \frac {m+3}2\right\}\cup\\
\cup\left\{4i,\,5j\colon 3\leq i\leq \frac m2,\, 3\leq j\leq \frac {m+3}2\right\}
\end{array}$ & Otherwise\\
$4152$ & $V(\mathcal{G}^m(H_{17}))\setminus\left\{1i,4j,4m,5j,6i,6(\frac {m+3}2) \colon\, 1\leq j\leq\frac {m+1}2\leq i<m\right\}$ & $m$ odd\\
& $V(\mathcal{G}^m(H_{17}))\setminus\left\{1i, 4j,4m, 5j, 6i,6({\frac {m+4}2})\colon\, 2<i\leq \frac {m+3}2,\, 1\leq j\leq \frac {m+3}2\right\}$ & Otherwise\\
$4162$ & $\left\{42,\,4m,\, 51,\, 52\right\}$ & For all $m$\\
$5161$ & $\left\{6m,\,1i,\,4j,\,5k,\,6i\colon\, 2\leq i\leq \frac {m+1}2,\,2\leq j<\frac {m+1}2,\,2\leq k\leq \frac {m+1}2\right\}$ & $m$ odd\\
& $\begin{array}{ll}
\left\{6m,\,1i,\,6j\colon\, 2\leq i\leq \frac m2,\, 2\leq j\leq \frac {m+3}2\right\}\\
\left\{4i,\,5i\colon\, 2\leq i\leq \frac m2\right\}
\end{array}$ & Otherwise\\
$5162$ & $\begin{array}{ll}
\left\{61,\,41,\,1i,\,6i\colon\, \frac {m+1}2<i\leq m,\, \frac {m+3}2<i\leq m\right\}\\
\cup\left\{4i,\,5j\colon\, \frac {m+1}2<i<m,\,\frac {m+1}2<j\leq m \right\}\end{array}$ & $m$ odd\\
& $\begin{array}{ll}
\left\{61,\,41,\,1i,\,6i\colon\, \frac {m+3}2<i\leq m\right\}\\
\cup\left\{4i,\,5j\colon \frac {m+3}2\leq  i<m,\, \frac {m+3}2< j\leq m\right\}
\end{array}$ & Otherwise\\
$6162$ & $\left\{11,\,4{\frac {m+1}2},\,51,\, 6{\frac {m+3}2}\right\}$ & $m$ odd\\
& $\begin{array}{ll}
\left\{11,\,1({\frac {m+3}2})\,51,\, 5({\frac {m+3}2})\right\}
\end{array}$ & Otherwise
\end{tabular}
\vspace{0.1cm}
 \caption{Resolving neighbourhoods of $\mathcal{G}^m(H_{17})$.}\label{Table6d}
\end{table}

\begin{table}[ht]
 \centering
\begin{tabular}{lc||lc}
  $G$ & Asymptotic behaviour & $G$ & Asymptotic behaviour\\ \hline
  $Q_1$ & Unbounded &   $H_1,\, H_2$ & $=1$\\
  $Q_2$ & Unbounded & $H_3$ & $\sim 2$\\
$P_1$ & Unknown & $H_4, H_{10}, H_{12}, H_{17}$ & $\sim 2$\\
$P_2$ & $\sim 2$ & $H_5, H_{11}, H_{13}$ & $\sim 2$\\
$P_3$ & $\sim 2$ & $H_6, H_9, H_{14}, H_{16}$  & Unknown\\
$P_4,\, P_5,\, P_6$ & $\sim 2$ & $H_7$  & $\sim 2$\\
& & $H_8$  & Unknown\\
& & $H_{15}$ & $\sim 2$
\end{tabular}
 \caption{Asymptotic behaviour of $\mathrm{ldim}_{\mathrm{f}}(\mathcal{G}^m(G))$.}\label{t1}
\end{table}

\end{document}